\documentclass[11pt]{article}

\usepackage{amssymb}
\usepackage{amsmath}
\usepackage{amsthm}
\usepackage{color}
\usepackage[all]{xy}
\usepackage[colorlinks=true,citecolor=black,linkcolor=black,urlcolor=blue]{hyperref}

\newcommand{\XYMATRIX}{\xymatrix@M=6pt}
\newcommand{\aremb}{\ar@{^{(}->}}
\newcommand{\arembfrom}{\ar@{<-^{)}}}

\numberwithin{equation}{section}
\def\labelenumi{(\roman{enumi})}

\theoremstyle{plain}
\newtheorem{THM}{Theorem}[section]
\newtheorem{LEM}[THM]{Lemma}
\newtheorem{COR}[THM]{Corollary}
\newtheorem{PROP}[THM]{Proposition}

\theoremstyle{definition}
\newtheorem{DEF}[THM]{Definition}
\newtheorem{EX}[THM]{Example}

\theoremstyle{remark}

\newcommand{\preclex}{\mathrel{\prec_{\mathit{lex}}}}

\renewcommand{\le}{\leqslant}
\renewcommand{\ge}{\geqslant}

\newcommand{\0}{\varnothing}

\renewcommand{\sec}{\cap}
\renewcommand{\phi}{\varphi}
\renewcommand{\epsilon}{\varepsilon}
\newcommand{\UNION}{\bigcup}

\newcommand{\CC}{\mathbf{C}}
\newcommand{\DD}{\mathbf{D}}
\newcommand{\FF}{\mathbf{F}}

\newcommand{\KK}{\mathbf{K}}

\newcommand{\QQ}{\mathbb{Q}}
\newcommand{\RR}{\mathbb{R}}

\newcommand{\ZZ}{\mathbb{Z}}
\newcommand{\union}{\cup}
\newcommand{\restr}[2]{\hbox{$#1$}\hbox{$\upharpoonright$}_{#2}}
\newcommand{\reduct}[2]{\hbox{$#1$}\hbox{$|$}_{#2}}
\newcommand{\select}[2]{#1:#2}
\newcommand{\Boxed}[1]{\mbox{$#1$}}

\newcommand{\id}{\mathrm{id}}

\newcommand{\Ob}{\mathrm{Ob}}

\newcommand{\arity}{\mathrm{ar}}

\newcommand{\Arr}{\mathrm{Arr}}
\newcommand{\spec}{\mathrm{spec}}
\newcommand{\cod}{\mathrm{cod}}
\newcommand{\canlongrightarrow}{\mathrel{\overset{\mathit{can}}{\longrightarrow}}}

\newcommand{\calA}{\mathcal{A}}
\newcommand{\calB}{\mathcal{B}}
\newcommand{\calC}{\mathcal{C}}
\newcommand{\calD}{\mathcal{D}}
\newcommand{\calE}{\mathcal{E}}
\newcommand{\calF}{\mathcal{F}}
\newcommand{\calG}{\mathcal{G}}
\newcommand{\calH}{\mathcal{H}}

\newcommand{\calM}{\mathcal{M}}

\newcommand{\calQ}{\mathcal{Q}}

\newcommand{\calS}{\mathcal{S}}

\newcommand{\calU}{\mathcal{U}}
\newcommand{\calV}{\mathcal{V}}
\newcommand{\calW}{\mathcal{W}}
\newcommand{\calX}{\mathcal{X}}
\newcommand{\calY}{\mathcal{Y}}

\newcommand{\bfr}{\overline{r}}
\newcommand{\bfs}{\overline{s}}

\newcommand{\CHemb}{\mathbf{Ch}}
\newcommand{\GraEmb}{\mathbf{Gra}}
\newcommand{\HGraEmb}{\mathbf{Hgr}}
\newcommand{\TourEmb}{\mathbf{Tour}}
\newcommand{\EDigEmb}{\mathbf{EDig}}
\newcommand{\EPosEmb}{\mathbf{Pos}}
\newcommand{\MetIso}{\mathbf{Met}}
\newcommand{\REL}{\mathbf{Rel}}
\newcommand{\OOGRAEmb}{\mathbf{OGra}}

\DeclareMathOperator{\dom}{dom}

\DeclareMathOperator{\Forb}{Forb}
\DeclareMathOperator{\tp}{tp}

\DeclareMathOperator{\mat}{mat}
\DeclareMathOperator{\tup}{tup}

\title{Canonizing structural Ramsey theorems}
\author{%
  Dragan Ma\v sulovi\'c\\
  University of Novi Sad, Faculty of Sciences\\
  Department of Mathematics and Informatics\\
  Trg Dositeja Obradovi\'ca 3, 21000 Novi Sad, Serbia\\
  e-mail: dragan.masulovic@dmi.uns.ac.rs}

\begin{document}
\maketitle

\begin{abstract}
  At the beginning of 1950's Erd\H os and Rado suggested the investigation of the Ramsey-type results where
  the number of colors is not finite. This marked the birth of the so-called canonizing Ramsey theory.
  In 1985 Pr\"omel and Voigt made the first step towards the structural canonizing Ramsey theory
  when they proved the canonical Ramsey property for the class of finite linearly ordered hypergraphs,
  and the subclasses thereof defined by forbidden substructures.
  Building on their results in this paper we provide several new structural
  canonical Ramsey results. We prove the canonical Ramsey theorem for the class of all finite linearly ordered
  tournaments, the class of all finite posets with linear extensions and the class of all finite linearly
  ordered metric spaces. We conclude the paper with the canonical version of the celebrated
  Ne\v set\v ril-R\"odl Theorem. In contrast to the ``classical'' Ramsey-theoretic approach, in this
  paper we advocate the use of category theory to manage the complexity of otherwise technically
  overwhelming proofs typical in canonical Ramsey theory.
  
  \bigskip


  \noindent \textbf{AMS Subj.\ Classification (2010):} 05C55, 18A99
\end{abstract}

\section{Introduction}

The leitmotif of Ramsey theory is to prove the existence of regular patterns that occur when a large structure
is considered in a restricted context. It started with the following result of F.~P.~Ramsey~\cite{Ramsey}:

\begin{THM}[Ramsey Theorem \cite{Ramsey}]
  For positive integers $k$ and $r$ and an arbitrary coloring $\chi : \binom \omega k \to \{1, 2, \ldots, r\}$
  there exists an infinite set $S \subseteq \omega$ such that $\chi(X) = \chi(Y)$ for all $X, Y \in \binom S k$.
\end{THM}

\noindent
Here, $\omega = \{0, 1, 2, \ldots\}$, and for a set $S$ and a positive integer $k$ by $\binom Sk$ we denote the
set of all the $k$-element subsets of~$S$. Its finite version takes the following form.

\begin{THM}[Finite Ramsey Theorem \cite{Ramsey}]
  For positive integers $k$, $m$ and $r$ there exists an integer $n$ such that for every coloring
  $\chi : \binom nk \to \{1, 2, \ldots, r\}$
  there exists a set $S \in \binom nm$ such that $\chi(X) = \chi(Y)$ for all $X, Y \in \binom S k$.
\end{THM}

Generalizing the Finite Ramsey Theorem, the structural Ramsey theory originated at
the beginning of 1970’s in a series of papers (see \cite{N1995} for references).
We say that a class $\KK$ of finite structures has the \emph{Ramsey property} if the following holds:
for any number $k \ge 2$ of colors and all $\calA, \calB \in \KK$ such that $\calA$ embeds into $\calB$
there is a $\calC \in \KK$
such that no matter how we color the copies of $\calA$ in $\calC$ with $k$ colors, there is a \emph{monochromatic} copy
$\calB'$ of $\calB$ in $\calC$ (that is, all the copies of $\calA$ that fall within $\calB'$ are colored by the same color).

Many natural classes of structures (such as finite graphs, metric spaces and posets, just to
name a few) do not have the Ramsey property. It is quite common,
though, that after expanding the structures under consideration
with appropriately chosen linear orders, the resulting class of expanded
structures has the Ramsey property. For example, the class of all finite
linearly ordered graphs $(V, E, \Boxed<)$, where $(V, E)$ is a finite graph and $<$ is
a linear order on the set $V$ of vertices of the graph, has the Ramsey property~\cite{AH,Nesetril-Rodl-1976}.
The same is true for metric spaces~\cite{Nesetril-metric}.
In case of finite posets the class of all the structures
$(P, \Boxed\sqsubseteq, \Boxed<)$ where $(P, \Boxed\sqsubseteq)$ is
a finite poset and $<$ is a linear order on $P$ which extends~$\sqsubseteq$ has the Ramsey property~\cite{PTW,fouche}.

One of the cornerstones of the structural Ramsey theory is the famous Ne\v set\v ril-R\"odl Theorem
whose formulation requires some terminology.
Let $\Theta = (R_i)_{i \in I}$ be a sequence of finitary relational symbols.
A \emph{linearly ordered $\Theta$-structure} $\calA = (A, \Theta^\calA, \Boxed{<^\calA})$
is a set $A$ together with a sequence $\Theta^\calA = (R_i^\calA)_{i \in I}$ of finitary relations on $A$ (which are
the interpretations of the symbols in $\Theta$), and with a linear order~$<^\calA$ on~$A$.
A finite linearly ordered $\Theta$-structure $\calA = (A, \Theta^\calA, \Boxed{<^\calA})$
is \emph{irreducible} if for every $a, b \in A$ such that $a \ne b$ there is an $i \in I$ and a tuple
$(x_1, x_2, \ldots, x_{r_i}) \in R_i^\calA$ such that $x_p = a$ and $x_q = b$ for some~$p, q \in \{1, \ldots, r_i\}$
(here, $r_i$ is the arity of $R_i$).
For a family $\FF$ of irreducible finite linearly ordered $\Theta$-structures let
$\Forb_{\Theta, \Boxed<}(\FF)$ denote the class of all finite linearly ordered $\Theta$-structures $\calA$ such that
no structure from $\FF$ embeds into $\calA$ (so, $\FF$ is the family of \emph{forbidden} substructures).

\begin{THM}[The Ne\v set\v ril-R\"odl Theorem \cite{Nesetril-Rodl,Nesetril-Rodl-1983,Nesetril-Rodl-1989}]\label{canrp.thm.NRT}
  Let $\Theta$ be an arbitrary sequence of finitary relational symbols and let $\FF$ be a family of
  irreducible finite linearly ordered $\Theta$-structures. Then $\Forb_{\Theta, \Boxed<}(\FF)$ has the Ramsey property.
\end{THM}

At the beginning of 1950's Erd\H os and Rado suggested the investigation of the Ramsey-type results where
the number of colors is not finite. Their paper~\cite{Erdos-Rado-1950} marked the birth of the so-called
canonizing Ramsey theory.
Before we state the famous Erd\H os-Rado Canonization Theorem let us introduce a piece of notation.
Take any (possibly empty) $Q \subseteq \{1, 2, \ldots, k\}$.
For a $k$-element set $X = \{x_1, x_2, \ldots, x_k\} \subseteq \omega$ where $x_1 < x_2 < \ldots < x_k$
let $\select XQ = \{x_q : q \in Q\}$. (Note that $\select X\0 = \0$.)

\begin{THM}[Erd\H os-Rado Canonization Theorem \cite{Erdos-Rado-1950}]
  For a positive integer $k$ and an arbitrary coloring $\chi : \binom \omega k \to \omega$ there exists an infinite
  set $S \subseteq \omega$ and a possibly empty set $Q \subseteq \{1, 2, \ldots, k\}$ such that for all
  $X, Y \in \binom S k$ we have:
  $\chi(X) = \chi(Y)$ if and only if $\select XQ = \select YQ$.
\end{THM}

\noindent
Its finite version takes the following form.

\begin{THM}[Finite Erd\H os-Rado Canonization Theorem]\label{canrp.thm.FERCT}
  For positive integers $k$ and $m$ there exists an integer $n$ such that for every
  coloring $\chi : \binom n k \to \omega$ there exists a
  set $S \in \binom nm$ and a possibly empty set $Q \subseteq \{1, 2, \ldots, k\}$ such that for all
  $X, Y \in \binom S k$ we have: $\chi(X) = \chi(Y)$ if and only if $\select XQ = \select YQ$.
\end{THM}

In 1985 Pr\"omel and Voigt proved the canonical Ramsey theorem for hypergraphs~\cite{Promel-Voigt-Gra-1985}.
Let $\bfr = (r_1, r_2, \ldots, r_k)$ be a finite sequence of positive integers.
A \emph{linearly ordered $\bfr$-hypergraph} is a structure $\calH = (H, E_1, E_2, \ldots, E_k, \Boxed<)$
where $H$ is a set of \emph{vertices} of $\calH$, $<$ is a linear order on $H$
and $E_i \subseteq \binom{H}{r_i}$ is a set of \emph{$r_i$-hyperedges} of $\calH$, $1 \le i \le k$.
A finite linearly ordered $\bfr$-hypergraph $\calH = (H, E_1, E_2, \ldots, E_k, \Boxed<)$
is \emph{irreducible} if for every $a, b \in H$ such that $a \ne b$ there is an $i \in \{1, 2, \ldots, k\}$ and a hyperedge
$e \in E_i$ such that $a, b \in e$.
For a family $\FF$ of irreducible finite linearly ordered $\bfr$-hypergraphs let
$\Forb_{\bfr}(\FF)$ denote the class of all finite linearly ordered $\bfr$-hypergraphs $\calH$ such that
no hypergraph from $\FF$ embeds into $\calH$.

Let $\calG$ and $\calH$ be finite linearly ordered $\bfr$-hypergraphs. By $\binom \calG \calH$ we denote the set of all the
induced linearly ordered subhypergraphs of $\calG$ that are isomorphic to $\calH$.
Now, let $Q \subseteq \{1, 2, \ldots, n(\calH)\}$ be a (possibly empty) set, where $n(\calH)$ is the
number of vertices of $\calH$, and let $\{w_1 < w_2 < \ldots < w_{n(\calH)}\}$ be the set of vertices of~$\calH$
(as a linearly ordered set). By $\select \calH Q$ we denote the subhypergraph of $\calH$ induced by
$\{w_q : q \in Q\}$. (Note that $\select \calH\0$ is the empty linearly ordered hypergraph $(\0, \0, \0)$.)

\begin{THM}[Canonical Ramsey Theorem for Hypergraphs~\cite{Promel-Voigt-Gra-1985}]\label{canrp.thm.HGRA-CRP}
  Let $\bfr$ be a finite sequence of positive integers and let $\FF$ be a family of
  irreducible finite linearly ordered $\bfr$-hypergraphs. Then $\Forb_{\bfr}(\FF)$ has the canonical Ramsey property.

  Explicitly, for any $\calH, \calE \in \Forb_{\bfr}(\FF)$ there exists a
  $\calG \in \Forb_{\bfr}(\FF)$ such that for every coloring $\chi : \binom \calG\calH \to \omega$ there exists an
  $\calE^* \in \binom \calG\calE$ and a (possibly empty) set $Q \subseteq \{1, 2, \ldots, n(\calH)\}$ such that
  for all $\calH', \calH'' \in \binom{\calE^*}{\calH}$ we have:
  $\chi(\calH') = \chi(\calH'')$ if and only if $\select{\calH'}{Q} = \select{\calH''}{Q}$.
\end{THM}

Linearly ordered $(2)$-hypergraphs are usually referred to as \emph{linearly ordered graphs}, while
linearly ordered $(t)$-hypergraphs for $t \ge 2$ are usually referred to as \emph{linearly ordered $t$-uniform hypergraphs}.
The following is an immediate consequence of Theorem~\ref{canrp.thm.HGRA-CRP}:

\begin{COR}[\cite{Promel-Voigt-Gra-1985}]\label{canrp.cor.HGRA-CRP}
  $(a)$ The class of all finite linearly ordered graphs has the canonical Ramsey property.

  $(b)$ For every $n \ge 3$ the class of all finite linearly ordered $K_n$-free graphs has the canonical Ramsey property.
  (Here, $K_n$ stands for the complete graph on $n$ vertices; a graph is \emph{$K_n$-free} if it does not embed~$K_n$.)

  $(c)$ The class of all finite linearly ordered $t$-uniform hypergraphs, $t \ge 2$, has the canonical Ramsey property.
\end{COR}

Theorem~\ref{canrp.thm.HGRA-CRP} appears to be the first structural canonical Ramsey result. In this paper we build on the results
of~\cite{Promel-Voigt-Gra-1985} to provide several new structural canonical Ramsey results.
In contrast to~\cite{Promel-Voigt-Gra-1985} where the authors prove canonical Ramsey statements using the ``classical''
Ramsey-theoretic approach, in this paper we modify the ideas from~\cite{masul-preadj} and using the appropriate
``transfer techniques'' formulated in the language of category theory
we prove the canonical Ramsey theorem for the class of all finite linearly ordered tournaments, the class of all finite
posets with linear extensions, the class of all finite linearly ordered metric spaces and
the class of all finite linearly ordered oriented graphs. We conclude the paper with the canonical version of the
celebrated Ne\v set\v ril-R\"odl Theorem.

In Section~\ref{canrp.sec.catth} we provide a brief overview of basic category-theoretic notions.
Section~\ref{canrp.sec.CRP-LCT} is devoted to the reinterpretation of the canonical Ramsey property in the
language of category theory. As the motivating example for our categorical techniques we derive
the canonical Ramsey property for the class of all finite linearly ordered tournaments.
In Section~\ref{canrp.sec.ple} we present a technical result which enables us to transfer the canonical Ramsey
property from a category to its hereditary subcategory, and as an immediate consequence prove the canonical Ramsey
property for the class of all finite posets with linear extensions.
In Section~\ref{canrp.sec.metspc} we introduce canonical pre-adjunctions between two categories (see~\cite{masul-preadj}
for the motivation) and use them to prove the canonical Ramsey property for the class of all finite linearly ordered metric spaces,
as well as some standard subclasses thereof. The paper concludes with Section~\ref{canrp.sec.cNRT}
in which we prove the canonical version of the Ne\v set\v ril-R\"odl Theorem (Theorem~\ref{canrp.thm.NRT})
and from it easily derive the canonical Ramsey property for the class of all finite linearly ordered oriented graphs.

\section{Categories and functors}
\label{canrp.sec.catth}

In this section we provide a brief overview of basic elementary category-theoretic notions.
For a detailed account of category theory we refer the reader to~\cite{AHS}.

In order to specify a \emph{category} $\CC$ one has to specify
a class of objects $\Ob(\CC)$, a set of morphisms $\hom_\CC(A, B)$ for all $A, B \in \Ob(\CC)$,
the identity morphism $\id_A$ for all $A \in \Ob(\CC)$, and
the composition of morphisms~$\cdot$~so that
$\id_B \cdot f = f = f \cdot \id_A$ for all $f \in \hom_\CC(A, B)$, and
$(f \cdot g) \cdot h = f \cdot (g \cdot h)$ whenever the compositions are defined.
A morphism $f \in \hom_\CC(B, C)$ is \emph{monic} or \emph{left cancellable} if
$f \cdot g = f \cdot h$ implies $g = h$ for all $g, h \in \hom_\CC(A, B)$ where $A \in \Ob(\CC)$ is arbitrary.

\begin{EX}\label{canrp.ex.CH-def}
  Let $\Theta = (R_i)_{i \in I}$ be a sequence of finitary relational symbols. Any class $\KK$
  of $\Theta$-structures can be thought of as a category whose objects are the objects from $\KK$ and whose
  morphisms are the embeddings. In particular:
  \begin{enumerate}\renewcommand{\labelenumi}{(\theenumi)}
  \item
    A \emph{chain} is a pair $(A, \Boxed<)$ where $<$ is a linear ($=$~total) order on~$A$.
    In case $A$ is finite, instead of $(A, \Boxed<)$ we shall simply write $A = \{a_1 < a_2 < \ldots < a_n\}$.
    We shall also allow chains to be empty. The \emph{empty chain} is, therefore, the structure $(\0, \0)$.
    Finite chains and embeddings constitute a category that we denote by~$\CHemb$.
  \item
    Finite linearly ordered graphs and embeddings constitute a category that we denote by~$\GraEmb$.
    We also allow the empty graph $(\0, \0, \0)$.
  \item
    Let $\bfr = (r_i)_{i \in I}$ be a sequence of positive integers.
    A \emph{linearly ordered $\bfr$-hypergraph} is a structure $\calH = (H, (E_i)_{i \in I}, \Boxed<)$
    where $H$ is a set of \emph{vertices} of $\calH$, $<$ is a linear order on $H$
    and $E_i \subseteq \binom{H}{r_i}$ is a set of \emph{$r_i$-hyperedges} of $\calH$, $i \in I$.
    Finite linearly ordered $\bfr$-hypergraphs and embeddings constitute a category that we denote by~$\HGraEmb(\bfr)$.
    We also allow the empty $\bfr$-hypergraph $(\0, (\0)_{i \in I}, \0)$.
  \item
    A \emph{reflexive digraph with a linear extension} is a structure $(V, \rho, \Boxed<)$ where $<$ is a linear order on $V$
    and $\rho \subseteq V^2$ is a reflexive binary relation such that $(x, y) \in \rho$ and $x \ne y$ implies
    $x < y$ for all $x, y \in V$.
    The \emph{empty reflexive digraph with a linear extension} is the structure $(\0, \0, \0)$.
    Finite reflexive digraphs with linear extensions together with
    embeddings constitute a category that we denote by~$\EDigEmb$.
  \item
    A \emph{linearly ordered tournament} is a structure $(V, E, \Boxed<)$ where $<$ is a linear order on~$V$ and
    $E \subseteq V^2$ is an irreflexive binary relation such that and for all $x, y \in V$ satisfying $x \ne y$
    we have that either $(x, y) \in E$ or $(y, x) \in E$.
    The \emph{empty tournament} is the structure $(\0, \0, \0)$.
    Finite linearly ordered tournaments and embeddings constitute a category that we denote by~$\TourEmb$.
  \item
    An \emph{oriented graph} $\calV = (V, \rho)$ is a set $V$ together with an irreflexive binary relation $\rho$
    on $V$ such that $(v_1, v_2) \in \rho \Rightarrow (v_2, v_1) \notin \rho$ for all $v_1, v_2 \in V$.
    A \emph{linearly ordered oriented graph} is a structure $\calV = (V, \rho, \Boxed{<})$
    where $(V, \rho)$ is an oriented graph and $<$ is a linear order on~$V$.
    The \emph{empty linearly ordered oriented graph} is the structure $(\0, \0, \0)$.
    Finite linearly ordered oriented graphs together with embeddings
    constitute a category which we denote by~$\OOGRAEmb$.
  \item
    A \emph{poset with a linear extension} is a structure $(A, \Boxed\sqsubseteq, \Boxed<)$
    where $<$ is a linear order on $V$ and $\Boxed\sqsubseteq \subseteq A^2$ is a partial order on~$A$
    (that is, a reflexive, antisymmetric and transitive relation) such that $x \sqsubseteq y$ and $x \ne y$ implies
    $x < y$ for all $x, y \in A$.
    The \emph{empty poset with a linear extension} is the structure $(\0, \0, \0)$.
    Finite posets with linear extensions and embeddings constitute a category that we denote by~$\EPosEmb$.
  \item
    A \emph{linearly ordered metric space} is a structure $\calM = (M, d, \Boxed<)$ where $d : M^2 \to \RR$ is a
    \emph{metric} and $<$ is a linear order on $M$. A linearly ordered metric space $(M, d, \Boxed<)$ is \emph{rational} if
    $d : M^2 \to \QQ$. The \emph{empty metric space} is the structure $(\0, \0, \0)$.
    Finite linearly ordered metric spaces and isometric embeddings constitute a category that we denote by~$\MetIso$.
  \item
    For a sequence $\Theta = (R_i)_{i \in I}$ of finitary relational symbols and a binary relational symbol $<$ not in
    $\Theta$ let $\REL(\Theta, \Boxed<)$ denote the category whose objects are
    all the finite linearly ordered $\Theta$-relational structures and whose morphisms are embeddings.
    We also allow the empty linearly ordered $\Theta$-relational structure $(\0, (\0)_{i \in I}, \0)$.
  \end{enumerate}
\end{EX}

A category $\DD$ is a \emph{subcategory} of a category $\CC$ if $\Ob(\DD) \subseteq \Ob(\CC)$ and
$\hom_\DD(A, B) \subseteq \hom_\CC(A, B)$ for all $A, B \in \Ob(\DD)$.
A category $\DD$ is a \emph{full subcategory} of a category $\CC$ if $\Ob(\DD) \subseteq \Ob(\CC)$ and
$\hom_\DD(A, B) = \hom_\CC(A, B)$ for all $A, B \in \Ob(\DD)$.
A category $\DD$ is a \emph{hereditary subcategory} of a category $\CC$ if $\DD$ is a full subcategory
of $\CC$ and for all $D \in \Ob(\DD)$ and all $C \in \Ob(\CC)$, if $\hom_\CC(C, D) \ne \0$ then $C \in \Ob(\DD)$.

\begin{EX}
  For a linearly ordered metric space $\calM = (M, d, \Boxed<)$ let
  $$
    \spec(\calM) = \{d(x, y) : x, y \in M\}
  $$
  denote the \emph{spectre} of $\calM$, that is, the set of all the distances that are attained by points in~$\calM$.
  For a nonempty finite $S \subseteq \RR$ of nonnegative reals let $\MetIso(S)$ denote the full subcategory of $\MetIso$
  spanned by all those $\calM \in \Ob(\MetIso)$ satisfying $\spec(\calM) \subseteq S$.
\end{EX}

A \emph{functor} $F : \CC \to \DD$ from a category $\CC$ to a category $\DD$ maps $\Ob(\CC)$ to
$\Ob(\DD)$ and maps morphisms of $\CC$ to morphisms of $\DD$ so that
$F(f) \in \hom_\DD(F(A), F(B))$ whenever $f \in \hom_\CC(A, B)$, $F(f \cdot g) = F(f) \cdot F(g)$ whenever
$f \cdot g$ is defined, and $F(\id_A) = \id_{F(A)}$.

Categories $\CC$ and $\DD$ are \emph{isomorphic} if there exist functors $F : \CC \to \DD$ and $G : \DD \to \CC$ which are
inverses of one another both on objects and on morphisms.

An \emph{oriented multigraph} $\Delta$ consists of a collection (possibly a class) of vertices $\Ob(\Delta)$,
a collection of arrows $\Arr(\Delta)$, and two maps $\dom, \cod : \Arr(\Delta) \to \Ob(\Delta)$ which
assign to each arrow $f \in \Arr(\Delta)$ its domain $\dom(f)$ and its codomain $\cod(f)$.
If $\dom(f) = \gamma$ and $\cod(f) = \delta$ we write briefly $f : \gamma \to \delta$.
Intuitively, an oriented multigraph is a ``category without composition''. Therefore,
each category $\CC$ can be understood as an oriented multigraph
whose vertices are the objects of the category and whose arrows are the morphisms of the category.
A \emph{multigraph homomorphism} between oriented multigraphs $\Gamma$ and $\Delta$
is a pair of maps (which we denote by the same symbol) $F : \Ob(\Gamma) \to \Ob(\Delta)$ and
$F : \Arr(\Gamma) \to \Arr(\Delta)$ such that if $f : \sigma \to \tau$ in $\Gamma$, then
$F(f) : F(\sigma) \to F(\tau)$ in $\Delta$.

Let $\CC$ be a category. For any oriented multigraph $\Delta$, a \emph{diagram in $\CC$ of shape $\Delta$}
is a multigraph homomorphism $F : \Delta \to \CC$. Intuitively, a diagram in $\CC$ is an
arrangement of objects and morphisms in $\CC$ that has the shape of~$\Delta$.
A diagram $F : \Delta \to \CC$ is \emph{commutative} if morphisms along every two paths between the same
nodes compose to give the same morphism.

A diagram $F : \Delta \to \CC$ is \emph{has a commutative cocone in $\CC$} if there exists a $C \in \Ob(\CC)$
and a family of morphisms $(e_\delta : F(\delta) \to C)_{\delta \in \Ob(\Delta)}$ such that for every
arrow $g : \delta \to \gamma$ in $\Arr(\Delta)$ we have $e_\gamma \cdot F(g) = e_\delta$:
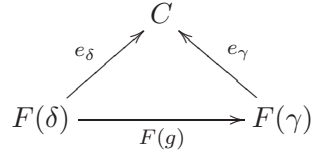
$$
  \xymatrix{
     & C & \\
    F(\delta) \ar[ur]^{e_\delta} \ar[rr]_{F(g)} & & F(\gamma) \ar[ul]_{e_\gamma}
  }
$$
(see Fig.~\ref{nrt.fig.3} for an illustration).
We say that $C$ together with the family of morphisms
$(e_\delta)_{\delta \in \Ob(\Delta)}$ is a \emph{commutative cocone in $\CC$ over the diagram~$F$}.

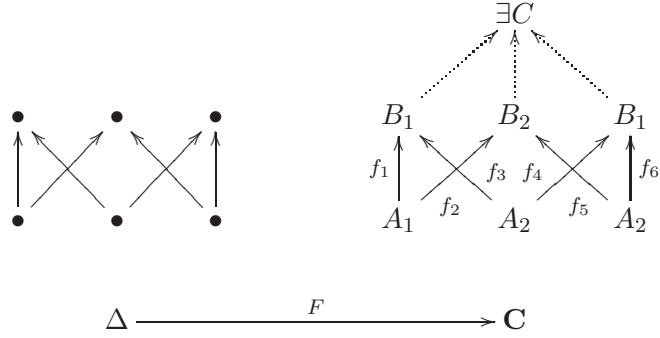
\begin{figure}
  $$
  \xymatrix{
    & & & & & \exists C &
  \\
    \bullet & \bullet & \bullet
    & & B_1 \ar@{.>}[ur] & B_2 \ar@{.>}[u] & B_1 \ar@{.>}[ul]
  \\
    \bullet \ar[u] \ar[ur] & \bullet \ar[ur] \ar[ul] & \bullet \ar[ul] \ar[u]
    & & A_1 \ar[u]^{f_1} \ar[ur]_(0.3){f_2} & A_2 \ar[ur]^(0.3){f_4} \ar[ul]_(0.3){f_3} & A_2 \ar[ul]^(0.3){f_5} \ar[u]_{f_6}
  \\
    & \Delta \ar[rrrr]^F  & & & & \CC  
  }
  $$
  \caption{A diagram in $\CC$ (of shape $\Delta$) with a commutative cocone}
  \label{nrt.fig.3}
\end{figure}

\section{The canonical Ramsey property in the language of category theory}
\label{canrp.sec.CRP-LCT}

For a set $\calS$ we say that
$
  \calS = \calX_0 \union \calX_1 \union \ldots \union \calX_k \union \ldots
$
is an \emph{$\omega$-coloring} of $\calS$ if $\calX_i \sec \calX_j = \0$ whenever $i \ne j$.
Equivalently, an $\omega$-coloring of $\calS$ is any mapping $\chi : \calS \to \omega$.
The relationship between the two notions is obvious and we shall use both.

\begin{DEF}
  For $A, B, C \in \Ob(\CC)$ we write
  $
    C \canlongrightarrow (B)^{A}
  $
  to denote that for every $\omega$-coloring
  $
    \chi : \hom_\CC(A, C) \to \omega
  $
  there is a morphism $w \in \hom_\CC(B, C)$, an object $Q \in \Ob(\CC)$ and a morphism
  $q \in \hom_\CC(Q, A)$ such that, for all $f, g \in \hom_\CC(A, B)$ we have:
  $\chi(w \cdot f) = \chi(w \cdot g)$ if and only if $f \cdot q = g \cdot q$.
  $$
    \XYMATRIX{
      Q \ar[r]^q & A \ar@/^3mm/[r]^f \ar@/_3mm/[r]_g & B \ar[r]^w & C
    }
  $$
  A category $\CC$ has the \emph{canonical Ramsey property} if
  for all $A, B \in \Ob(\CC)$ such that $\hom_\CC(A, B) \ne \0$
  there is a $C \in \Ob(\CC)$ such that $C \canlongrightarrow (B)^{A}$.
\end{DEF}

\begin{EX}\label{canrp.ex.CRP-ch}
  The category $\CHemb$ of finite chains and embeddings (Example~\ref{canrp.ex.CH-def}) has the canonical
  Ramsey property. This is just a reformulation of the Finite Erd\H os-Rado Canonization Theorem
  (Theorem~\ref{canrp.thm.FERCT}). To see that this is indeed the case, it suffices to note that
  $\select XQ$ corresponds to the image of the embedding $q \mapsto x_q$
  of the finite chain $(Q, \Boxed<)$ into the finite chain $(X, \Boxed<)$.
\end{EX}

\begin{EX}\label{canrp.ex.CRP-gra}
  The category $\GraEmb$ of finite linearly ordered graphs and embeddings
  has the canonical Ramsey property. This is just a reformulation of Corollary~\ref{canrp.cor.HGRA-CRP}~$(a)$.
\end{EX}

Let $\CC$ be a category and let $\FF \subseteq \Ob(\CC)$ be a class of objects in~$\CC$.
By $\Forb_\CC(\FF)$ we denote the full subcategory of $\CC$ spanned by the class of all those
$A \in \Ob(\CC)$ satisfying $\hom_\CC(F, A) = \0$ for all $F \in \FF$. The class
$\FF$ is then referred to as a class of \emph{forbidden subobjects}.

\begin{EX}\label{canrp.ex.CRP-hgra}
  Let $\bfr = (r_1, r_2, \ldots, r_n)$ be a finite sequence of positive integers.
  The category $\HGraEmb(\bfr)$ of finite linearly ordered $\bfr$-hypergraphs and embeddings
  has the canonical Ramsey property. Moreover, if $\FF$ is a family of irreducible
  finite linearly ordered $\bfr$-hypergraphs then the category $\Forb_{\HGraEmb(\bfr)}(\FF)$
  has the canonical Ramsey property. This is just a reformulation of Theorem~\ref{canrp.thm.HGRA-CRP}.
\end{EX}

Clearly, if $\CC$ and $\DD$ are isomorphic categories, then one of them has the canonical Ramsey property if and only if
the other one does. This is the easiest way to transfer the the canonical Ramsey property from one
category to the other.

\begin{PROP}\label{canrp.prop.edig-otur-CRP}
  $(a)$ The category $\EDigEmb$ has the canonical Ramsey property.

  $(b)$ The category $\TourEmb$ has the canonical Ramsey property.
\end{PROP}
\begin{proof}
  $(a)$ Let us show that $\EDigEmb$ has the canonical Ramsey property by showing that
  the categories $\EDigEmb$ and $\GraEmb$ are isomorphic. The claim then follows from Example~\ref{canrp.ex.CRP-gra}.
  Consider the functors $F : \GraEmb \to \EDigEmb$ and $G : \EDigEmb \to \GraEmb$ defined by
  \begin{align*}
    F(V, E, \Boxed<) &= (V, \rho_E, \Boxed<), \; F(f) = f, \text{ and}\\
    G(V, \rho, \Boxed<) &= (V, E_\rho, \Boxed<), \; G(f) = f,
  \end{align*}
  where
  \begin{align*}
    \rho_E &= \{(x, y) : \{x, y\} \in E \text{ and } x < y\} \union \{(x, x) : x \in V \}, \text{ and}\\
    E_\rho &= \{\{x, y\} : (x, y) \in \rho \text{ and } x \ne y\}.
  \end{align*}
  It is easy to see that $F$ and $G$ are well defined and that they are mutually inverse. Therefore,
  the categories $\EDigEmb$ and $\GraEmb$ are isomorphic.

  $(b)$ Let us show that $\TourEmb$ has the canonical Ramsey property by showing that this category is also
  isomorphic to $\GraEmb$. Consider the functors $F : \GraEmb \to \TourEmb$ and $G : \TourEmb \to \GraEmb$ defined by
  \begin{align*}
    F(V, E, \Boxed<) &= (V, E', \Boxed<), \; F(f) = f, \text{ and}\\
    G(V, T, \Boxed<) &= (V, T', \Boxed<), \; G(f) = f,
  \end{align*}
  where
  \begin{align*}
    E' &= \{(x, y) : \{x, y\} \in E \text{ and } x < y\} \union \{(x, y) : \{x, y\} \notin E \text{ and } x > y\}, \text{ and}\\
    T' &= \{\{x, y\} : (x, y) \in T \text{ and } x < y\}.
  \end{align*}
  It is easy to see that $F$ and $G$ are well defined and that they are mutually inverse. Therefore,
  the categories $\TourEmb$ and $\GraEmb$ are isomorphic.
\end{proof}

\section{Posets with linear extensions}
\label{canrp.sec.ple}

Another way of transferring the Ramsey property is from a category to its subcategory.
We shall now present a technical result which enables us to transfer the canonical Ramsey
property from a category to its hereditary subcategory.

Consider a finite, acyclic, bipartite digraph where all the arrows go from one class of vertices into the other
and the out-degree of all the vertices in the first class is~2:
$$
  
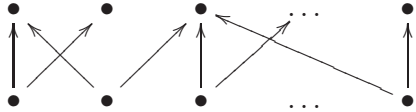
\xymatrix{
    \bullet & \bullet & \bullet & \ldots & \bullet \\
    \bullet \ar[u] \ar[ur] & \bullet \ar[ur] \ar[ul] & \bullet \ar[u] \ar[ur] & \ldots & \bullet \ar[u] \ar[ull]
  }
$$
\noindent
Such a digraph will be referred to as a \emph{binary digraph}.
A \emph{binary diagram} in a category $\CC$ is a diagram $F : \Delta \to \CC$ where $\Delta$ is a binary digraph,
$F$ takes the bottom row of $\Delta$ onto the same object, and takes the top row of $\Delta$ onto
the same object, Fig.~\ref{nrt.fig.2}.
A subcategory $\DD$ of a category $\CC$ is \emph{closed for binary diagrams} if every binary diagram
$F : \Delta \to \DD$ which has a commuting cocone in $\CC$ has a commuting cocone in~$\DD$.

\begin{figure}
  $$
  \xymatrix{
    \bullet & \bullet & \bullet
    & & \calB & \calB & \calB
  \\
    \bullet \ar[u] \ar[ur] & \bullet \ar[ur] \ar[ul] & \bullet \ar[ul] \ar[u]
    & & \calA \ar[u]^{f_1} \ar[ur]_(0.3){f_2} & \calA \ar[ur]^(0.3){f_4} \ar[ul]_(0.3){f_3} & \calA \ar[ul]^(0.3){f_5} \ar[u]_{f_6}
  \\
    & \Delta \ar[rrrr]^F  & & & & \CC  
  }
  $$
  \caption{A binary diagram in $\CC$ (of shape $\Delta$)}
  \label{nrt.fig.2}
\end{figure}
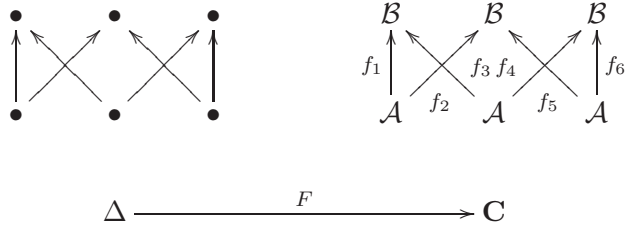

\begin{THM}\label{canrp.thm.1}
  Let $\CC$ be a category such that every morphism in $\CC$ is monic and
  such that $\hom_\CC(\calA, \calB)$ is finite for all $\calA, \calB \in \Ob(\CC)$, and let $\DD$ be a hereditary
  subcategory of~$\CC$. If $\CC$ has the canonical Ramsey property and $\DD$ is closed for binary diagrams,
  then $\DD$ has the canonical Ramsey property.
\end{THM}
\begin{proof}
  Take any $A, B \in \Ob(\DD)$ such that $\hom_\DD(A, B) \ne \0$. Since $\DD$ is a subcategory of $\CC$
  and $\CC$ has the canonical Ramsey property, there is a $C \in \Ob(\CC)$ such that
  $C \canlongrightarrow (B)^A$.

  Let us now construct a binary diagram in $\DD$ as follows. Let $\hom_\CC(B, C) = \{e_1, e_2, \ldots, e_n\}$.
  Intuitively, for each $e_i \in \hom_\CC(B, C)$ we add a copy of $B$ to the diagram, and whenever $e_i \cdot u = e_j \cdot v$
  for some $u, v \in \hom_\DD(A, B)$ we add a copy of $A$ to the diagram together with two arrows:
  one going into the $i$th copy of $B$ labelled by $u$ and another one going into the $j$th copy of $B$ labelled by~$v$
  (note that, by the construction, this diagram has a commuting cocone in $\CC$):
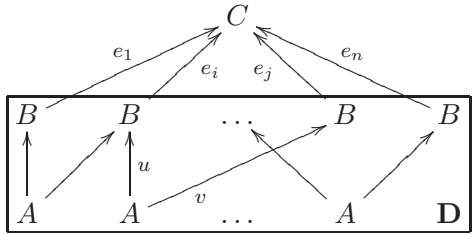
  $$
  \xymatrix{
    & & C
  \\
    B \ar[urr]^{e_1} & B \ar[ur]_(0.6){e_i} & \ldots & B \ar[ul]^(0.6){e_j} & B \ar[ull]_{e_n}
  \\
    A \ar[u] \ar[ur] & A \ar[urr]_(0.3){v} \ar[u]_{u} & \ldots & A \ar[ur] \ar[ul] & \DD
    \save "2,1"."3,5"*[F]\frm{} \restore
  }
  $$
  Formally, let $\Delta$ be the binary digraph whose objects are
  \begin{align*}
    \Ob(\Delta) = \{1, 2, \ldots, n\} \union \{(u, v, i, j) : \; & 1 \le i,j \le n; \; i \ne j;\\
                  &u, v \in \hom_\DD(\calA, \calB); \; e_i \cdot u = e_j \cdot v\}
  \end{align*}
  and whose arrows are of the form $u : (u, v, i, j) \to i$ and $v : (u, v, i, j) \to j$.
  Let $F : \Delta \to \DD$ be the following diagram whose action on objects is:
  \begin{align*}
    F(i) &= B, && 1 \le i \le n,\\
    F((u, v, i, j)) &= A, && e_i \cdot u = e_j \cdot v,
  \end{align*}
  and whose action on morphisms is $F(g) = g$:
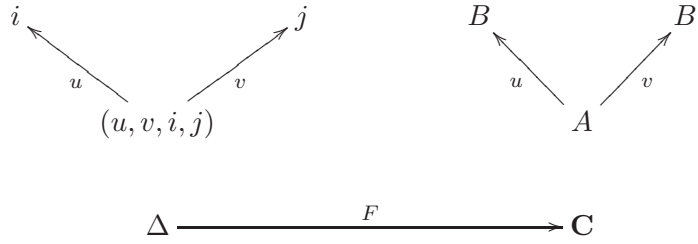
  $$
  \xymatrix{
    i &  & j
    & & B &  & B
  \\
      & (u, v, i, j) \ar[ur]_v \ar[ul]^u &  
    & &   & A \ar[ur]_v \ar[ul]^u & 
  \\
    & \Delta \ar[rrrr]^F  & & & & \CC  
  }
  $$

  As we have already observed in the informal discussion above, the diagram $F : \Delta \to \DD$ has a commuting cocone in $\CC$,
  so, by the assumption, it has a commuting cocone in~$\DD$. Therefore, there is
  a $D \in \Ob(\DD)$ and morphisms $f_i : B \to D$, $1 \le i \le n$, such that the following diagram in $\DD$ commutes:
  $$
  
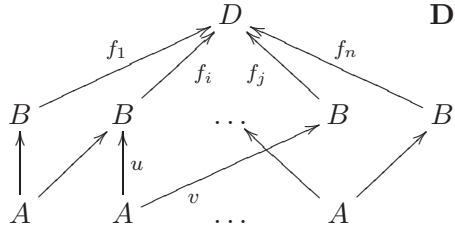
\xymatrix{
    & & D & & \DD
  \\
    B \ar[urr]^{f_1} & B \ar[ur]_(0.6){f_i} & \ldots & B \ar[ul]^(0.6){f_j} & B \ar[ull]_{f_n}
  \\
    A \ar[u] \ar[ur] & A \ar[urr]_(0.3){v} \ar[u]_{u} & \ldots & A \ar[ur] \ar[ul] & 
  }
  $$
  Let us show that in $\DD$ we have $D \canlongrightarrow (B)^A$. Take any $\omega$-coloring
  $$
    \hom_\DD(A, D) = \calX_0 \union \calX_1 \union \ldots \union \calX_k \union \ldots
  $$
  and define an $\omega$-coloring
  $$
    \hom_\CC(A, C) = \calX'_0 \union \calX'_1 \union \ldots \union \calX'_k \union \ldots
  $$
  as follows. For $j \in \omega$ let
  $$
    \calX'_{j+1} = \{e_s \cdot u : 1 \le s \le n, u \in \hom_\DD(A, B), f_s \cdot u \in \calX_j \},
  $$
  and then let
  $$
    \calX'_0 = \hom_\CC(A, C) \setminus \UNION_{j \in \omega} \calX'_{j+1}.
  $$
  Let us show that $\calX'_i \sec \calX'_j = \0$ whenever $i \ne j$. By the definition of $\calX'_0$
  it suffices to consider the case where $i \ge 1$ and $j \ge 1$. 
  Assume, to the contrary, that there is an $h \in \calX'_{i+1} \sec \calX'_{j+1}$ for some $i, j \in \omega$
  such that~$i \ne j$.
  Then $h = e_s \cdot u$ for some $s$ and some $u \in \hom_\DD(A, D)$ such that $f_s \cdot u \in \calX_i$, and
  $h = e_t \cdot v$ for some $t$ and some $v \in \hom_\DD(A, D)$ such that $f_t \cdot v \in \calX_j$.
  Then $e_s \cdot u = h = e_t \cdot v$. Clearly, $s \ne t$ and we have that $(u, v, s, t) \in \Ob(\Delta)$.
  (Suppose, to the contrary, that $s = t$. Then
  $e_s \cdot u = e_s \cdot v$ implies $u = v$ because all the morphisms in $\CC$ are monic.
  But then $\calX_i \ni f_s \cdot u = f_s \cdot v = f_t \cdot v \in \calX_j$, which contradicts the assumption that
  $\calX_i \sec \calX_j = \0$.) Consequently,
  $f_s \cdot u = f_t \cdot v$ because $D$ and the morphisms $f_i : B \to D$, $1 \le i \le n$, form
  a commuting cocone over $F : \Delta \to \DD$ in~$\DD$. Therefore, $f_s \cdot u = f_t \cdot v \in \calX_i \sec \calX_j$,
  which is not possible.

  Let $\chi : \hom_\DD(A, D) \to \omega$ be the coloring such that $\chi(\calX_i) = i$ for all $i \in \omega$, and let
  $\chi' : \hom_\CC(A, C) \to \omega$ be the coloring such that $\chi'(\calX'_i) = i$ for all $i \in \omega$.
  Since, by construction, $C \canlongrightarrow (B)^A$, there is an $e_\ell \in \hom_\CC(B, C)$, an object
  $Q \in \Ob(\CC)$ and a morphism $q \in \hom_\CC(Q, A)$ such that
  $$
    \chi'(e_\ell \cdot u) = \chi'(e_\ell \cdot v) \text{ if and only if } u \cdot q = v \cdot q,
  $$
  for all $u, v \in \hom_\CC(A, B)$.
  By the assumption, $\DD$ is a hereditary subcategory of $\CC$ so $Q \in \Ob(\DD)$ and
  $q \in \hom_\DD(Q, A)$. Finally, in order to show that
  $$
    \chi(f_\ell \cdot u) = \chi(f_\ell \cdot v) \text{ if and only if } u \cdot q = v \cdot q
  $$
  for all $u, v \in \hom_\DD(A, B) = \hom_\CC(A, B)$ it suffices to note that
  $\chi(f_\ell \cdot u) = \chi'(e_\ell \cdot u) - 1$ for all $u \in \hom_\DD(A, B)$, and that
  $\chi'(e_\ell \cdot u) \ge 1$ for all $u \in \hom_\DD(A, B)$.
\end{proof}

\begin{COR}\label{canrp.cor.EPosEmb}
  The category $\EPosEmb$ has the canonical Ramsey property.
\end{COR}
\begin{proof}
  Clearly, morphisms in $\EDigEmb$ are monic, hom-sets in $\EDigEmb$ are finite and
  $\EDigEmb$ has the canonical Ramsey property (see Proposition~\ref{canrp.prop.edig-otur-CRP}).
  Since $\EPosEmb$ is a hereditary subcategory of~$\EDigEmb$, in order to prove that
  $\EPosEmb$ has the canonical Ramsey property it suffices to show that $\EPosEmb$ is closed for binary diagrams
  in~$\EDigEmb$ (Theorem~\ref{canrp.thm.1}).

  Let $F : \Delta \to \EPosEmb$ be a binary diagram in $\EPosEmb$ where the top row of the diagram maps onto
  $\calB = (B, \Boxed{\sqsubseteq^\calB}, \Boxed{<^\calB})$ and the bottom row of the diagram maps onto
  $\calA = (A, \Boxed{\sqsubseteq^\calA}, \Boxed{<^\calA})$. Assume that $F$ has
  a commuting cocone in $\EDigEmb$ with the tip at $\calC = (C, \rho^\calC, \Boxed{<^\calC})$ and the morphisms
  $e_1, \ldots, e_n$:
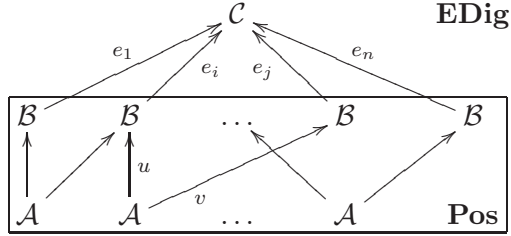
  $$
    \xymatrix{
      & & \calC & & \EDigEmb
    \\
      \calB \ar[urr]^{e_1} & \calB \ar[ur]_(0.6){e_i} & \ldots & \calB \ar[ul]^(0.6){e_j} & \calB \ar[ull]_{e_n}
    \\
      \calA \ar[u] \ar[ur] & \calA \ar[urr]_(0.3){v} \ar[u]_{u} & \ldots & \calA \ar[ur] \ar[ul] & \EPosEmb
      \save "2,1"."3,5"*[F]\frm{} \restore
    }
  $$
  Define $\calD = (D, \rho^\calD, \Boxed{<^\calD})$ as follows: $D = e_1(B) \union e_2(B) \union \ldots \union e_n(B)
  \subseteq C$,
  $\rho^\calD$ is the transitive closure of $\rho^\calC \sec D^2$, and $\Boxed{<^\calD} = \Boxed{<^\calC} \sec D^2$.
  The relation $\rho^\calD$ is clearly reflexive (because $\rho^\calC$ is reflexive) and transitive.
  It is also easy to see that $(x, y) \in \rho^\calD$ and $x \ne y$ implies $x \mathrel{<^\calD} y$.
  Therefore, $\rho^\calD$ is a partial order on $D$ and $<^\calD$ is a linear extension of~$\rho^\calD$,
  so $\calD \in \Ob(\EPosEmb)$.
  
  For $i \in \{1, 2, \ldots, n\}$ define $f_i : B \to D$ by $f_i(x) = e_i(x)$, $x \in B$.
  It is easy to see that each $f_i$ is actually an embedding $\calB \to \calD$. Therefore,
  $\calD$ together with the embeddings $f_1, \ldots, f_n$ forms a commuting cocone over $F$ in $\EPosEmb$.
  This completes the proof that $\EPosEmb$ is closed for binary diagrams in $\EDigEmb$.
\end{proof}

\section{Metric spaces}
\label{canrp.sec.metspc}

One useful strategy for proving the Ramsey property for categories
consists of establishing a pre-adjunction between two categories (see~\cite{masul-preadj}).
As the canonical Ramsey property is much stronger than the ``usual'' Ramsey property, we shall need a
stronger version which we refer to as a \emph{canonical pre-adjunction}.

\begin{DEF}\label{canrp.def.CPA}
  Let $\CC$ and $\DD$ be categories. A pair of maps
  $$
    F : \Ob(\DD) \rightleftarrows \Ob(\CC) : G
  $$
  is a \emph{canonical pre-adjunction between $\CC$ and $\DD$} provided there is a family of maps
  $$
    \Phi_{\calY,\calX} : \hom_\CC(F(\calY), \calX) \to \hom_\DD(\calY, G(\calX))
  $$
  and a family of maps
  $$
    F_{\calY,\calX} : \hom_\DD(\calY, \calX) \to \hom_\CC(F(\calY), F(\calX))
  $$
  satisfying the following (when appropriate we shall omit the subscripts for the family of maps $F$ and treat
  $F$ as a functor-like entity):
  \begin{itemize}
  \item[(CPA1)]
  for every $\calC \in \Ob(\CC)$, every $\calD, \calE \in \Ob(\DD)$,
  every $u \in \hom_\CC(F(\calD), \calC)$ and every $f \in \hom_\DD(\calE, \calD)$ we have that
  $\Phi_{\calD, \calC}(u) \cdot f = \Phi_{\calE, \calC}(u \cdot F(f))$.
  $$
    \XYMATRIX{
      F(\calD) \ar[rr]^u                             & & \calC & & \calD \ar[rr]^{\Phi_{\calD, \calC}(u)}                    & & G(\calC) \\
      F(\calE) \ar[u]^{F(f)} \ar[urr]_{u \cdot F(f)} & &       & & \calE \ar[u]^f \ar[urr]_{\;\;\;\;\Phi_{\calE, \calC}(u \cdot F(f))}
    }
  $$
  \item[(CPA2)]
  for all $\calD, \calE \in \Ob(\DD)$ and $\calQ \in \Ob(\CC)$, and for every $q \in \hom_\CC(\calQ, F(\calE))$
  there exist a $\calQ' \in \Ob(\DD)$ and a $q' \in \hom_\DD(\calQ', \calE)$ such that for all $f, g \in \hom_\DD(\calE, \calD)$
  we have: $F(f) \cdot q = F(g) \cdot q$ if and only if $f \cdot q' = g \cdot q'$.
  $$
    \XYMATRIX{
      F(\calD)                                                                             & & \calD \\
      F(\calE) \ar@/^3mm/[u]^{F(f)} \ar@/_3mm/[u]_{F(g)} & & \calE \ar@/^3mm/[u]^f \ar@/_3mm/[u]_g \\
      \calQ \ar[u]^q                                                                       & & \calQ' \ar[u]_{q'}
    }
  $$
  \end{itemize}
\end{DEF}
(Note that in a pre-adjunction $F$ and $G$ are \emph{not} required to be functors, although $F$ is ``defined on morphisms''
as well.)

\begin{THM}\label{canrp.thm.CPA}
  Let $\CC$ and $\DD$ be categories such that $\CC$ has the canonical Ramsey property.
  If there exists a canonical pre-adjunction $F : \Ob(\DD) \rightleftarrows \Ob(\CC) : G$
  then $\DD$ has the canonical Ramsey property.
\end{THM}
\begin{proof}
  Let $F : \Ob(\DD) \rightleftarrows \Ob(\CC) : G$ be a canonical pre-adjunction and let
  $$
    \Phi_{\calY,\calX} : \hom_\CC(F(\calY), \calX) \to \hom_\DD(\calY, G(\calX))
  $$
  and
  $$
    F_{\calY,\calX} : \hom_\DD(\calY, \calX) \to \hom_\CC(F(\calY), F(\calX))
  $$
  be families of maps satisfying (CPA1) and (CPA2).

  Take any $\calD, \calE \in \Ob(\DD)$. Since $\CC$ has the canonical Ramsey property, there is a $\calC \in \Ob(\CC)$
  such that $\calC \canlongrightarrow (F(\calD))^{F(\calE)}$. Let us show that $G(\calC) \canlongrightarrow (\calD)^{\calE}$.
  Take any coloring $\chi : \hom_\DD(\calE, G(\calC)) \to \omega$ and construct
  a coloring $\chi' : \hom_\CC(F(\calE), \calC) = \omega$ as follows:
  \begin{equation}\label{canrp.eq.1}
    \chi'(u) = \chi(\Phi_{\calE, \calC}(u)).
  \end{equation}
  By the choice of $\calC$ there exist a $u \in \hom_\CC(F(\calD), \calC)$, a $\calQ \in \Ob(\CC)$
  and a $q \in \hom_\CC(\calQ, F(\calE))$ such that
  \begin{equation}\label{canrp.eq.2}
    \chi'(u \cdot \alpha) = \chi'(u \cdot \beta) \text{ iff } \alpha \cdot q = \beta \cdot q,
  \end{equation}
  for all $\alpha, \beta \in \hom_\CC(F(\calE), F(\calD))$.
  By (CPA2) there exist a $\calQ' \in \Ob(\DD)$ and a $q' \in \hom_\DD(\calQ', \calE)$ such that for all
  $f, g \in \hom_\DD(\calE, \calD)$ we have:
  \begin{equation}\label{canrp.eq.22}
    F(f) \cdot q = F(g) \cdot q  \text{ iff } f \cdot q' = g \cdot q'.
  \end{equation}
  Let us show that for all $f, g \in \hom_\DD(\calE, \calD)$:
  \begin{equation*}
    \chi(\Phi_{\calD, \calC}(u) \cdot f) = \chi(\Phi_{\calD, \calC}(u) \cdot g) \text{ iff } f \cdot q' = g \cdot q'.
  \end{equation*}
  This follows as a sequence of straightforward equivalences:
  \begin{align*}
    &\chi(\Phi_{\calD, \calC}(u) \cdot f) = \chi(\Phi_{\calD, \calC}(u) \cdot g)\\
    &\text{\quad iff } \chi(\Phi_{\calE, \calC}(u \cdot F(f))) = \chi(\Phi_{\calE, \calC}(u \cdot F(g))) && \text{by (CPA1)}\\
    &\text{\quad iff } \chi'(u \cdot F(f)) = \chi'(u \cdot F(g)) && \text{by \eqref{canrp.eq.1}}\\
    &\text{\quad iff } F(f) \cdot q = F(g) \cdot q && \text{by \eqref{canrp.eq.2}}\\
    &\text{\quad iff } f \cdot q' = g \cdot q' && \text{by \eqref{canrp.eq.22}}
  \end{align*}
  which completes the proof.
\end{proof}

As a demonstration of this strategy we shall show that the class of all finite linearly ordered metric spaces
has the canonical Ramsey property. The proof is a modification of the proof of~\cite[Theorem~4.4]{masul-preadj} and
the technical results that we inherit from~\cite{masul-preadj} shall not be repeated here.

Let $T = \{0 = t_0 < t_1 < \ldots < t_\ell \} \subseteq \RR$ be a finite set of nonnegative reals. We say that $T$ is
\emph{tight}~\cite{masul-preadj} if $t_{i + j} \le t_i + t_j$ for all $0 \le i \le j \le i + j \le \ell$.

\begin{THM}\label{canrp.thm.met}
  $(a)$ The category $\MetIso(S)$ has the canonical Ramsey property for every finite
  tight set $S = \{0 = s_0 < s_1 < \ldots < s_k \} \subseteq \RR$.

  $(b)$ Let $(A, \Boxed+)$ be a subsemigroup of the additive semigroup $(\RR, \Boxed+)$
  such that $0 \in A$ and let $I$ be an arbitrary interval of reals.
  Then $\MetIso(\{0\} \union (I \sec A))$ has the canonical Ramsey property.

  $(c)$ The categories $\MetIso$, $\MetIso(\QQ)$ and $\MetIso(\ZZ)$ have the canonical Ramsey property.
\end{THM}
\begin{proof}
  $(a)$
  Let $S = \{0 = s_0 < s_1 < \ldots < s_k \} \subseteq \RR$ be a tight set.
  In order to show that $\MetIso(S)$ has the canonical Ramsey property it suffices to establish a
  canonical pre-adjunction
  $$
    F : \Ob(\MetIso(S)) \rightleftarrows \Ob(\EPosEmb) : G
  $$
  since $\EPosEmb$ has the canonical Ramsey property (Corollary~\ref{canrp.cor.EPosEmb}).

  For $\calM = (M, d, \Boxed<) \in \Ob(\MetIso(S))$ put
  $$
    F(\calM) = (M \times \{0, 1, \ldots, k\}, \Boxed\sqsubseteq, \Boxed\prec),
  $$
  where
  $$
    (x, i) \sqsubseteq (y, j) \text{ if and only if } i \le j \text{ and } d(x, y) \le s_j - s_i,
  $$
  and
  $$
    (x, i) \prec (y, j) \text{ if and only if } i < j, \text{ or } i = j \text{ and } x < y.
  $$
  It is easy to show (see~\cite{masul-preadj}) that
  $(M \times \{0, 1, \ldots, k\}, \Boxed\sqsubseteq, \Boxed\prec)$ is a poset with a linear extension, so the definition
  of $F$ is correct.
    
  For $\calA = (A, \Boxed\sqsubseteq, \Boxed\prec) \in \Ob(\EPosEmb)$ put
  $$
    A^{<k} = \{(a_0, a_1, \ldots, a_{k-1}) : a_i \in A, 0 \le i \le k - 1\}.
  $$
  Define $d_\calA : (A^{<k})^2 \to S$ as follows:
  $$
    d_\calA(\overline a, \overline b) = s_j
  $$
  where
  $$
    j = \min\{ p \in \{0, 1, \ldots, k-1\} : (\forall i \le k - 1 - p)(a_i \sqsubseteq b_{i + p} \land b_i \sqsubseteq a_{i + p}) \},
  $$
  and $\min \0 = k$. Next, put $\overline a \preclex \overline b$ if and only if there is a $j$ such that
  $a_j \prec b_j$ and $(\forall i < j)(a_i = b_i)$. Finally, let
  $$
    G(\calA) = (A^{<k}, d_\calA, \preclex).
  $$
  It was shown in~\cite{masul-preadj} that $(A^{<k}, d_\calA, \preclex)$ is a linearly ordered metric
  space with distances in $S$.
  
  For $\calM = (M, d, \Boxed<) \in \Ob(\EPosEmb(S))$ and $\calA = (A, \Boxed\sqsubseteq, \Boxed\prec) \in \Ob(\EPosEmb)$
  let us define
  $$
    \Phi_{\calM, \calA} : \hom_{\EPosEmb}(F(\calM), \calA) \to \hom_{\MetIso(S)}(\calM, G(\calA))
  $$
  as follows. For $u : F(\calM) \hookrightarrow \calA$ let $\hat u = \Phi_{\calM, \calA}(u) : M \to A^{<k}$ be defined by
  $$
    \hat u (x) = (u(x, 0), u(x, 1), \ldots, u(x, k - 1)).
  $$
  It was shown in~\cite{masul-preadj} that the definition of $\Phi$ is correct, that is,
  for every $u : F(\calM) \hookrightarrow \calA$  the mapping $\hat u$ is an embedding $\calM \hookrightarrow G(\calA)$.

  As the final ingredient of the canonical pre-adjunction we are constructing, for finite linearly ordered
  metric spaces $\calM = (M, d, \Boxed<)$ and $\calM' = (M', d', \Boxed<)$ let us define
  $$
    F_{\calM', \calM} : \hom_{\MetIso(S)}(\calM', \calM) \to \hom_{\EPosEmb}(F(\calM'), F(\calM))
  $$
  by
  $$
    F_{\calM', \calM}(f)(x, i) = (f(x), i).
  $$
  It was shown in~\cite{masul-preadj} that $F_{\calM', \calM}(f)$ is an embedding $F(\calM') \hookrightarrow F(\calM)$.
  In the sequel we shall omit the subscripts in $F_{\calM', \calM}$.
  
  We still have to show that these families of maps satisfy the requirements~(CPA1) and~(CPA2)
  of Definition~\ref{canrp.def.CPA}.
  
  (CPA1) Let us show that $\Phi_{\calM, \calA}(u) \circ f = \Phi_{\calM', \calA}(u \circ F(f))$.
  Put $\hat u = \Phi_{\calM, \calA}(u)$ and $\widehat{u \circ F(f)} = \Phi_{\calM', \calA}(u \circ F(f))$. Then
  \begin{align*}
    \widehat{u \circ F(f)}(x)
      &= \big(  u \circ F(f)(x, 0), u \circ F(f)(x, 1), \ldots, u \circ F(f)(x, k - 1)  \big)\\
      &= \big(  u(f(x), 0), u(f(x), 1), \ldots, u(f(x), k - 1)  \big)\\
      &= \hat u (f(x)) = \hat u \circ f (x).
  \end{align*}

  (CPA2) Take any $\calQ = (Q, \Boxed\sqsubseteq, \Boxed\prec) \in \Ob(\EPosEmb)$ and assume that
  $\calQ \le F(\calM')$ so that $q : \calQ \to F(\calM')$ is the inclusion $x \mapsto x$.
  Let
  $$
    Q = \{(x_1, r_1), (x_2, r_2), \ldots, (x_p, r_p)\} \subseteq M' \times \{0, 1, \ldots, k\}.
  $$
  Let $Q' = \{x_1, x_2, \ldots, x_p\} \subseteq M'$ and take $\calQ' = (Q', d', \Boxed<)$
  to be the subspace of $\calM'$ induced by~$Q'$. Now
  take any $f, g \in \hom_{\MetIso(S)}(\calQ', \calM')$ and note that
  $$
    \restr{f}{Q'} = \restr{g}{Q'} \text{ if and only if } \restr{F(f)}{Q} = \restr{F(g)}{Q}
  $$
  holds trivially.
  
  $(b)$ If $\{0\} \union (I \sec A) = \{0\}$ the statement is trivially true. Assume, therefore, that
  $\{0\} \union (I \sec A) \ne \{0\}$. Then $A \ne \{0\}$.
  Take any $\calU, \calV \in \Ob(\MetIso(\{0\} \union (I \sec A)))$ such that $\calU \hookrightarrow \calV$.
  Since $\calV$ is finite, $S = \spec(\calV)$ is a finite subset of~$A$.

  By \cite[Lemma~4.3]{masul-preadj}, there exists a finite tight set
  $T = \{0 = t_0 < t_1 < \ldots < t_\ell \} \subseteq A$ such that $S \subseteq T$, $t_1 = s_1$ and $t_\ell = s_k$.
  Then $\calU, \calV \in \Ob(\MetIso(T))$ because
  $\spec(\calU) \subseteq \spec(\calV) = S \subseteq T$. The category $\MetIso(T)$ has the canonical
  Ramsey property by $(a)$, so
  there is a $\calW \in \Ob(\MetIso(T))$ such that $\calW \canlongrightarrow (\calV)^\calU$.
  Since, by construction, the smallest and the largest nonzero
  elements of $S$ and $T$ coincide and since $S \subseteq \{0\} \union (I \sec A)$
  it follows that $T \subseteq \{0\} \union (I \sec A)$, so $\MetIso(T)$ is a full subcategory of $\MetIso(\{0\} \union (I \sec A))$
  whence $\calW \in \Ob(\MetIso(\{0\} \union (I \sec A)))$.

  $(c)$ Directly from $(b)$.
\end{proof}

In~\cite{masul-preadj} we used the same strategy based on pre-adjunctions and a very similar argument to prove that the class
of all finite convexly ordered ultrametric spaces has the Ramsey property (see~\cite{masul-preadj} for technical details).
Interestingly, the generalization we outlined here fails to provide the analogous result that
the class of all finite convexly ordered ultrametric spaces has the canonical Ramsey property.
To the best of our knowledge, the status of the canonical Ramsey property for the class of all
finite convexly ordered ultrametric spaces is still an open problem.

\section{Canonizing the Ne\v set\v ril-R\"odl Theorem}
\label{canrp.sec.cNRT}

We shall now prove the canonical version of the Ne\v set\v ril-R\"odl Theorem (Theorem~\ref{canrp.thm.NRT}).
Unsurprisingly, our starting point is Theorem~\ref{canrp.thm.HGRA-CRP}. What remains to be done is to translate
the context of the Ne\v set\v ril-R\"odl Theorem (formulated in terms of finite relational structures)
to the context of hypergraphs.

Theorem~\ref{canrp.thm.HGRA-CRP} shows that $\Forb_{\bfr}(\FF)$ has the canonical Ramsey property for every family
$\FF$ of irreducible finite linearly ordered $\bfr$-hypergraphs and every \emph{finite} sequence $\bfr$
of positive integers. Let us now prove a ``sideways generalization''
of this result where $\bfr$ is no longer required to be finite at the cost of stipulating that
$\FF$ be finite.

Let $\bfr = (r_i)_{i \in I}$ be an arbitrary (not necessarily finite) sequence of positive integers and let
$\calH = (H, (E_i^{\calH})_{i \in I}, \Boxed{<^{\calH}})$ be a linearly ordered $\bfr$-hypergraph.
Take any $I_0 \subseteq I$ and let $\bfr_0 = (r_{i_0})_{i_0 \in I_0}$.
Then the linearly ordered $\bfr_0$-hypergraph $\reduct{\calH}{I_0} = (H, (E_{i_0}^{\calH})_{i_0 \in I_0}, \Boxed{<^{\calH}})$
will be referred to as the \emph{$I_0$-reduct of~$\calH$}. Clearly, if $f : \calH \hookrightarrow \calG$ is an embedding
between linearly ordered $\bfr$-hypergraphs $\calH$ and $\calG$, then
$f : \reduct{\calH}{I_0} \hookrightarrow \reduct{\calG}{I_0}$ is also an embedding for every $I_0 \subseteq I$.

Let $\bfs = (s_k)_{j \in J}$ be another arbitrary (not necessarily finite) sequence of positive integers and let
$g : I \to J$ be a surjective map. For a linearly ordered $\bfs$-hypergraph
$\calH^0 = (H, (E_j^{\calH^0})_{j \in J}, \Boxed{<^{\calH^0}})$ define a linearly ordered $\bfr$-hypergraph
$\calH = (H, (E_i^{\calH})_{i \in I}, \Boxed{<^{\calH^0}})$ on the same set of vertices and with the same linear
ordering so that $E_i^\calH = E_{g(i)}^{\calH^0}$, for all $i \in I$.
We then call $\calH$ the \emph{$g$-polymer of $\calH^0$}.

\begin{LEM}\label{canrp.lem.polymer}
  Let $\bfr = (r_i)_{i \in I}$ and $\bfs = (s_j)_{j \in J}$ be arbitrary (not necessarily finite) sequences of positive integers
  and let $g : I \to J$ be a surjective map. Let $\calH^0$ and $\calG^0$ be linearly ordered $\bfs$-hypergraphs and let
  $\calH$ and $\calG$ be the $g$-polymers of $\calH^0$ and $\calG^0$, respectively. Then $f : \calH \to \calG$ is an embedding
  if and only of $f : \calH^0 \to \calG^0$ is an embedding.
\end{LEM}
\begin{proof}
  Obvious.
\end{proof}

\begin{THM}\label{canrp.thm.HGRA-CRP-2}
  Let $\bfr = (r_i)_{i \in I}$ be an arbitrary (not necessarily finite) sequence of positive integers and let
  $\FF = \{\calF_1, \calF_2, \ldots, \calF_m\}$ be a finite family of irreducible finite linearly ordered $\bfr$-hypergraphs.
  Then $\Forb_{\HGraEmb(\bfr)}(\FF)$ has the canonical Ramsey property.
\end{THM}
\begin{proof}
  Fix an arbitrary sequence $\bfr = (r_i)_{i \in I}$ of positive integers and
  a finite family $\FF = \{\calF_1, \calF_2, \ldots, \calF_m\}$ of irreducible finite linearly ordered $\bfr$-hypergraphs
  where $\calF_j = (F_j, (E_i^{\calF_j})_{i \in I}, \Boxed{<^{\calF_j}})$, $1 \le j \le m$.
  Take any $\calA =(A, (E_i^{\calA})_{i \in I}, \Boxed{<^{\calA}})$ and $\calB = (B, (E_i^{\calB})_{i \in I}, \Boxed{<^{\calB}})$
  from $\Ob(\Forb_{\HGraEmb(\bfr)}(\FF))$ such that $\calA \hookrightarrow \calB$. Without loss of generality we can assume that
  all the sets $A$, $B$, $F_1$, \ldots, $F_m$ are pairwise disjoint.

  Let $\calH = (H, (E_i^{\calH})_{i \in I}, \Boxed{<^{\calH}})$ be the disjoint union of linearly ordered
  $\bfr$-hypergraphs $\calA$, $\calB$, $\calF_1$, \ldots, $\calF_m$ so that
  \begin{align*}
    H &= A \union B \union F_1 \union \ldots \union F_m,\\
    E_i^{\calH} &= E_i^\calA \union E_i^\calB \union E_i^{\calF_1} \union \ldots \union E_i^{\calF_m}, \text{ for each } i \in I, \text{ and }\\
    \Boxed{<^{\calH}} &= \Boxed{<^\calA} \oplus \Boxed{<^\calB} \oplus \Boxed{<^{\calF_1}} \oplus \ldots \oplus \Boxed{<^{\calF_m}}.
  \end{align*}
  Here, for disjoint linearly ordered sets $(L, \Boxed<)$ and $(M, \Boxed\sqsubset)$ by $\Boxed< \oplus \Boxed\sqsubset$
  we denote the linear order on $L \union M$ where every element of $L$
  is smaller then every element of $M$, the elements in $L$ are ordered linearly by~$<$, and the elements in $M$
  are ordered linearly by~$\sqsubset$.

  Clearly, $E_i^\calH = \0$ whenever $r_i > |H|$. On the other hand,
  for every $r_i \le |H|$ there are only finitely many
  possibilities to choose $E_i^\calH \subseteq \binom{H}{r_i}$.
  Therefore, there exists a \emph{finite} set $I_0 \subseteq I$ such that
  for every $i \in I$ we have that $E_i^\calH = E_{i_0}^\calH$ for some $i_0 \in I_0$.
  Define a surjective map $g : I \to I_0$ as follows:
  \begin{itemize}
  \item
    for $i_0 \in I_0$ put $g(i_0) = i_0$;
  \item
    for $i \in I \setminus I_0$ choose any $i_0 \in I_0$ such that $E_i^\calH = E_{i_0}^\calH$ and put
    $g(i) = i_0$.
  \end{itemize}
  Since $\calG \hookrightarrow \calH$ for every $\calG \in \{\calA, \calB, \calF_1, \ldots, \calF_m\}$,
  we have that
  \begin{equation}\label{canrp.eq.polymer}
    E_i^\calG = E_{g(i)}^\calG \text{ for every } \calG \in \{\calA, \calB, \calF_1, \ldots, \calF_m\}
    \text{ and every } i \in I.
  \end{equation}
  Now, let $\bfr_0 = (r_i)_{i \in I_0}$ and let $\calA^0$, $\calB^0$ and $\calF^0_j$, $1 \le j \le m$,
  be the $I_0$-reducts of $\calA$, $\calB$ and $\calF_j$, $1 \le j \le m$, respectively.
  Note that, by \eqref{canrp.eq.polymer},
  we also have that $\calA$, $\calB$ and $\calF_j$, $1 \le j \le m$, are $g$-polymers of
  $\calA^0$, $\calB^0$ and $\calF^0_j$, $1 \le j \le m$, respectively.

  Let us show that each $\calF^0_j$, $1 \le j \le m$, is irreducible. Take any $j \in \{1, \ldots, m\}$
  and distinct $a, b \in F_j$. Since $\calF_j$ is irreducible, there exists and $i \in I$ and a
  hyperedge $e \in E^{\calF_j}_i$ such that $a, b \in e$. But $E^{\calF_j}_i = E^{\calF_j}_{g(i)}$
  and $g(i) \in I_0$. Therefore, $e \in E^{\calF_j}_{g(i)}$ is a hyperedge of $\calF^0_j$.

  Since $\bfr_0$ is a finite sequence of positive integers and
  $\FF^0 = \{\calF^0_1, \calF^0_2, \ldots, \calF^0_m\}$ is a family of irreducible
  linearly ordered $\bfr_0$-hypergraphs, $\Forb_{\HGraEmb(\bfr_0)}(\FF^0)$ has the canonical Ramsey property by
  Theorem~\ref{canrp.thm.HGRA-CRP}.
  
  It is easy to see that $\calA^0, \calB^0 \in \Ob(\Forb_{\HGraEmb(\bfr_0)}(\FF^0))$. Namely,
  if $f : \calF^0_j \hookrightarrow \calA^0$ is an embedding for some~$j$ then
  $f : \calF_j \hookrightarrow \calA$ is also an embedding because of Lemma~\ref{canrp.lem.polymer} and the fact that
  $\calF_j$ and $\calA$ are $g$-polymers of $\calF_j^0$ and $\calA^0$, respectively.
  This contradicts the choice of~$\calA$. The same argument applies to $\calB^0$.

  Therefore, there is a $\calC^0 = (C, (E_{i_0}^{\calC^0})_{i_0 \in I_0}, \Boxed{<^{\calC^0}}) \in
  \Ob(\Forb_{\HGraEmb(\bfr_0)}(\FF^0))$ such that
  $\calC^0 \canlongrightarrow (\calB^0)^{\calA^0}$.
  Let $\calC = (C, (E_{i}^{\calC})_{i \in I}, \Boxed{<^{\calC^0}})$ be the $g$-polymer of $\calC^0$.
  As above, we easily conclude that $\calC \in \Forb_{\HGraEmb(\bfr)}(\FF)$. So, let us show that
  $\calC \canlongrightarrow (\calB)^\calA$.
  
  Take any $\omega$-coloring $\chi : \hom_{\HGraEmb(\bfr)}(\calA, \calC) \to \omega$. Since
  $\calA$, $\calB$ and $\calC$ are $g$-polymers of $\calA^0$, $\calB^0$ and $\calC^0$, respectively,
  it follows from Lemma~\ref{canrp.lem.polymer} that
  \begin{align*}
    \hom_{\HGraEmb(\bfr)}(\calA, \calC) &= \hom_{\HGraEmb(\bfr_0)}(\calA^0, \calC^0),\\
    \hom_{\HGraEmb(\bfr)}(\calA, \calB) &= \hom_{\HGraEmb(\bfr_0)}(\calA^0, \calB^0), \text{ and}\\
    \hom_{\HGraEmb(\bfr)}(\calB, \calC) &= \hom_{\HGraEmb(\bfr_0)}(\calB^0, \calC^0).
  \end{align*}
  Therefore, $\chi : \hom_{\HGraEmb(\bfr_0)}(\calA^0, \calC^0) \to \omega$ is a well-defined $\omega$-coloring.

  Since $\calC^0 \canlongrightarrow (\calB^0)^{\calA^0}$ in $\HGraEmb(\bfr_0)$,
  there is a $w \in \hom_{\HGraEmb(\bfr_0)}(\calB^0, \calC^0)$, an object $\calQ^0 \in \Ob(\HGraEmb(\bfr_0))$
  and a morphism $q \in \hom_{\HGraEmb(\bfr_0)}(\calQ^0, \calA^0)$ such that for all
  $f, g \in \hom_{\HGraEmb(\bfr_0)}(\calA^0, \calB^0)$ we have:
  $\chi(w \cdot f) = \chi(w \cdot g)$ if and only if $f \cdot q = g \cdot q$.
  Let $\calQ$ be the $g$-polymer of~$\calQ^0$. Then $\calQ \in \Ob(\HGraEmb(\bfr))$
  and $q \in \hom_{\HGraEmb(\bfr)}(\calQ, \calA)$ by Lemma~\ref{canrp.lem.polymer}, so
  for all $f, g \in \hom_{\HGraEmb(\bfr)}(\calA, \calB)$ we have:
  $\chi(w \cdot f) = \chi(w \cdot g)$ if and only if $f \cdot q = g \cdot q$.
  This completes the proof.
\end{proof}

\begin{THM}[Canonical Ne\v set\v ril-R\"odl Theorem]\label{canrp.thm.cNRT}
  Let $\Theta = (R_i)_{i \in I}$ be a sequence of finitary relational symbols and let $<$
  be a binary relational symbol not in $\Theta$. Let $\FF \subseteq \Ob(\REL(\Theta, \Boxed<))$ be a set
  consisting of irreducible linearly ordered $\Theta$-relational structures.
  If at least one of the sets $I$, $\FF$ is finite
  then $\Forb_{\REL(\Theta, \Boxed<)}(\FF)$ has the canonical Ramsey property.
\end{THM}
\begin{proof}
  We start by recalling some basic facts about total quasiorders.
  A total quasiorder is a reflexive and transitive binary relation such that each pair of elements
  of the underlying set is comparable. Each total quasiorder $\sigma$ on a set $I$ induces an equivalence relation $\equiv_\sigma$
  on $I$ and a linear order $\sqsubset_\sigma$ on $I / \Boxed{\equiv_\sigma}$ in a natural way: $i \mathrel{\equiv_\sigma} j$ if
  $(i, j) \in \sigma$ and $(j, i) \in \sigma$, and $(i / \Boxed{\equiv_\sigma}) \mathrel{\sqsubset_\sigma} (j / \Boxed{\equiv_\sigma})$
  if $(i, j) \in \sigma$ and $(j, i) \notin \sigma$.

  Let $(A, \Boxed<)$ be a linearly ordered set, let $r$ be a positive integer, let $I = \{1, \ldots, r\}$ and
  let $\overline a = (a_1, \ldots, a_r) \in A^r$. Then
  $$
    \tp(\overline a) = \{(i, j) : a_i \le a_j \}
  $$
  is a total quasiorder on $I$ which we refer to as the \emph{type} of~$\overline a$.
  Assume that $\sigma = \tp(\overline a)$. Let $s = |I / \Boxed{\equiv_\sigma}|$ and let $i_1$, \ldots, $i_s$
  be the representatives of the classes of $\equiv_\sigma$ enumerated so that
  $(i_1 / \Boxed{\equiv_\sigma}) \mathrel{\sqsubset_\sigma} \ldots \mathrel{\sqsubset_\sigma} (i_s / \Boxed{\equiv_\sigma})$.
  Then
  $$
    \mat(\overline a) = \{a_{i_1}, \ldots, a_{i_s}\}
  $$
  is the \emph{matrix of $\overline a$}. Conversely, given a matrix and a total
  quasiorder we can always reconstruct the original tuple as follows. For a total quasiorder $\sigma$ on $I$ such that
  $|I / \Boxed{\equiv_\sigma}| = s$ and an $s$-element set $\mu = \{b_1, \ldots, b_s\} \in \binom As$ such that $b_1 < \ldots < b_s$
  define an $r$-tuple
  $$
    \tup(\sigma, \mu) = (a_1, \ldots, a_r) \in A^r
  $$
  as follows. Let $i_1$, \ldots, $i_s$ be the representatives of the classes of $\equiv_\sigma$ enumerated so that
  $(i_1 / \Boxed{\equiv_\sigma}) \mathrel{\sqsubset_\sigma} \ldots \mathrel{\sqsubset_\sigma} (i_s / \Boxed{\equiv_\sigma})$.
  Then put
  $$
    a_\eta = b_\xi \text{ if and only if } \eta \mathrel{\equiv_\sigma} i_\xi.
  $$
  (In other words, we put $b_1$ on all the entries in $i_1 / \Boxed{\equiv_\sigma}$, we put
  $b_2$ on all the entries in $i_2 / \Boxed{\equiv_\sigma}$, and so on.) Then it is a matter of routine to check that
  \begin{equation}\label{nrt.eq.tup-mat}
    \begin{aligned}
      \tp(\tup(\sigma, \mu)) = \sigma, \text{ } \mat(\tup(\sigma, \mu))    &= \mu, \text{ and}\\
      \tup(\tp(\overline a), \mat(\overline a)) &= \overline a,\\
    \end{aligned}
  \end{equation}
  
  Now, for each $i \in I$ let $\Sigma_i$ be the set of all the total quasiorders on
  $\{1, 2, \ldots, \arity(R_i)\}$. Let
  $
    J = \bigcup_{i \in I} \{i\} \times \Sigma_i,
  $
  and for $j = (i, \sigma) \in J$ let $s_j = |\{1, 2, \ldots, \arity(R_i)\} / \Boxed{\equiv_\sigma}|$.
  Finally, put $\bfs = (s_j)_{j \in J}$.
  
  For $\calA = (A, \Theta^\calA, \Boxed{<^\calA}) \in \Ob(\REL(\Theta, \Boxed{<}))$
  define a $\calA^\dagger = (A, (E_j^{\calA^\dagger})_{j \in J}, \Boxed{<^{\calA}}) \in \Ob(\HGraEmb(\bfs))$ as follows:
  $$
    E_{(i, \sigma)}^{\calA^\dagger} = \{\mat(\overline a): \overline a \in R_i^\calA \text{ and } \tp(\overline a) = \sigma \}.
  $$
  On the other hand, take any
  $\calB = (B, (E_j^\calB)_{j \in J}, \Boxed{<^\calB}) \in \Ob(\HGraEmb(\bfs))$
  and define $\calB^* = (B, \Theta^{\calB^*}, \Boxed{<^{\calB}}) \in \Ob(\REL(\Theta, \Boxed<))$ as follows:
  $$
    R_i^{\calB^*} = \{\tup(\sigma, \mu) : \sigma \in \Sigma_i \text{ and } \mu \in E_{(i, \sigma)}^\calB\}.
  $$
  Because of \eqref{nrt.eq.tup-mat} we have that $(\calA^\dagger)^* = \calA$ and $(\calB^*)^\dagger = \calB$ for all
  $\calA \in \Ob(\REL(\Theta, \Boxed{<}))$ and all $\calB \in \Ob(\HGraEmb(\bfs))$.
  Therefore, the functor
  $$
    H : \REL(\Theta, \Boxed{<}) \to \HGraEmb(\bfs) : \calA \mapsto \calA^\dagger : f \mapsto f
  $$
  is an isomorphism between the categories $\REL(\Theta, \Boxed{<})$ and $\HGraEmb(\bfs)$, its inverse being
  $$
    G : \HGraEmb(\bfs) \to \REL(\Theta, \Boxed{<}) : \calB \mapsto \calB^* : f \mapsto f.
  $$
  Consequently, the categories $\Forb_{\REL(\Theta, \Boxed{<})}(\FF)$ and $\Forb_{\HGraEmb(\bfs)}(H(\FF))$ are isomorphic,
  the isomorphisms being the adequate restrictions of $H$ and~$G$.
  
  Clearly, if $I$ is a finite set then $J$ is also a finite set. So, at least one of the sets
  $J$, $H(\FF)$ is finite, where $H(\FF) = \{H(F) : F \in \FF\}$. It is also easy to see that $H(F)$ is an
  irreducible $\bfs$-hypergraph for each $F \in \FF$ (since each $F \in \FF$ is irreducible).
  Therefore, Theorems~\ref{canrp.thm.HGRA-CRP} and~\ref{canrp.thm.HGRA-CRP-2} imply that the category
  $\Forb_{\HGraEmb(\bfs)}(H(\FF))$ has the canonical Ramsey property. Since, as we have just seen,
  $\Forb_{\HGraEmb(\bfs)}(H(\FF))$ is isomorphic to $\Forb_{\REL(\Theta, \Boxed{<})}(\FF)$, the latter category
  also has the canonical Ramsey property.
\end{proof}

\begin{COR}
  The category $\OOGRAEmb$ has the canonical Ramsey property.
\end{COR}
\begin{proof}
  Let $R$ be a binary relational symbol and let $\mathbf 2 = (R)$.
  Let $\calF_1$ and $\calF_2$ be the following linearly ordered $\mathbf 2$-relational structures:
  \begin{align*}
    \calF_1 &= (\{1\}, \{(1,1)\}, \Boxed<), \text{ and}\\
    \calF_2 &= (\{1, 2\}, \{(1,2), (2,1)\}, \Boxed<),
  \end{align*}
  where $<$ is the usual ordering of the integers. Then it is easy to see that
  $\Forb_{\REL(\mathbf 2, \Boxed{<})}(\{\calF_1, \calF_2\}) = \OOGRAEmb$. Since
  $\calF_1$ and $\calF_2$ are irreducible, Theorem~\ref{canrp.thm.cNRT} yields that
  $\Forb_{\REL(\mathbf 2, \Boxed{<})}(\{\calF_1, \calF_2\}) = \OOGRAEmb$ has the canonical Ramsey property.
\end{proof}

\section{Acknowledgements}

The author gratefully acknowledges the support of the Grant No.\ 174019 of the Ministry of Education, Science and Technological Development of the Republic of Serbia.

\end{document}